\documentclass[11pt]{amsart}
\usepackage{amsmath,amssymb,amsfonts,amsthm, amscd,indentfirst}
\usepackage{amsmath,latexsym,amssymb,amsmath,%awheremscd,
	amscd,amsthm,amsxtra}
\usepackage{hyperref}\usepackage{url}
\usepackage{color}
\usepackage{pdfpages}
\usepackage{graphics}
\usepackage{enumitem}

\usepackage{marvosym}

\usepackage{graphicx}
\usepackage{caption}

\usepackage{epsfig,here}
\usepackage{subfigure,here}

\newtheorem{definition}{Definition}[section]
\newtheorem{theorem}{Theorem}[section]
\newtheorem{prop}{Proposition}[section]

\newtheorem{corollary}{Corollary}[section]

\newtheorem{remark}{\textbf{Remark}}[section]

\def\rr{\mathbb{R}}

\def\hh{\mathbb{H}}

\def\a{\alpha}

\def\<{\langle}
\def\>{\rangle}

\def\n{\nabla}

\numberwithin{equation} {section}

\begin{document}
	
	\title[A mean curvature type flow with capillary boundary in a horoball]{A mean curvature type flow with capillary boundary in a horoball in hyperbolic space}
		\author{Jinyu Guo}

%\affiliation{Department of Mathematical Sciences, Tsinghua University, Beijing, 100084, China}
\address{Department of Mathematical Sciences, Tsinghua University, Beijing, 100084, China}

%\corauth{Corresponding author}

%\thanks{*Corresponding author: guojinyu14@163.com}

%\address{\rm
%\newline
%Jinyu Guo
%\newline
%Department of Mathematical Sciences
%\newline
%Tsinghua University
%\newline
%Beijing 100084
%\newline
%People's Republic of China
%\newline
%e-mail: guojinyu14@163.com}
\email{guojinyu14@163.com}

%\thanks{The author was supported by China Postdoctoral Science Foundation (No.2022M720079) and Shuimu Tsinghua Scholar Program (No.2022SM046).}

	\begin{abstract}
		In this paper, we study a mean curvature type flow with capillary boundary in a horoball in hyperbolic space. Our flow preserves the volume of the bounded domain enclosed by the hypersurface and monotonically decreases the energy functional. We show that it has the long time existence and converges to a truncated umbilical hypersurface in hyperbolic space. As an application, we solve an isoperimetric type problem for hypersurfaces with capillary boundary in a horoball.\qquad

\end{abstract}
	
	%\date{}
	
	\maketitle
	
	\medskip

	%\tableofcontents
	
	%{\bf Keywords:} Capillary surfaces, Stability, Rigidity, Surface with free boundary.
\section{Introduction}
	
Guan-Li in \cite{GL0} first introduced a local constrained mean curvature type flow for closed hypersurfaces in an Euclidean space $\mathbb{R}^{n+1}$ as follows
\begin{equation*}
\begin{cases}{}
\partial_{t} x=(n-H\langle x, \nu\rangle) \nu & \text { in } M \times[0, T), \\
x(\cdot,0)=x_{0}& \text { in } M,
\end{cases}
\end{equation*}
where $H$ is mean curvature of hypersurface $M$. In view of classical Minkowski formula, the enclosed volume is preserved while the area is decreasing along the above flow. They showed that the flow exists for all time and converges to a geodesic ball provided the initial hypersurface is star-shaped. This kind of flow has been further explored by Guan-Li \cite{GL0} and Guan-Li-Wang \cite{GLW} in space forms and more general warped product spaces. Some related results can be seen in \cite{GL1, GL5, HL, HLW, WX1} and the references therein.

A hypersurface is said to be with a capillary boundary in a domain $B$ if the hypersurface intersects its support $\partial B$ at a constant contact angle. In particular, the capillary boundary reduces to free boundary when the contact angle is orthogonal.

In \cite{WX3} Wang-Xia proposed a locally constrained Guan-Li type mean curvature type flow with free boundary in an Euclidean ball. They gave a new concept of star-shaped free boundary hypersurfaces in a ball and showed the long time existence and convergence of the flow to a free boundary spherical cap provided the initial hypersurface is star-shaped. %further extended their research took charge of
In the sequence, Wang-Weng \cite{WW} extended their results to the hypersurfaces with capillary boundary and the contact angle $\theta$ satisfying $|\cos\theta|<\frac{3n+1}{5n-1}$ by using the Minkowski-type formulas in a ball.

In \cite{QWX} Qiang-Weng-Xia considered a Guan-Li type flow with free boundary in a geodesic ball in hyperbolic space. They also proved the flow exists for all time and converges to a spherical cap with free boundary if the initial hypersurface is star-shaped. Inspiring from the Wang-Weng's proof, Mei-Weng \cite{MW} recently generalized their results to certain capillary boundary case in geodesic ball in space forms. There are other interesting results about the hypersurfaces with capillary boundary, see for example \cite{MWW, WWX, SWX, WengX, WX2}.

%When $B$ is an Euclidean half-space space $\mathbb{R}^{n+1}_{+}$, Mei-Wang-Weng \cite{MWW} considered a Guan-Li type flow with $\theta$-capillary boundary in $\mathbb{R}^{n+1}_{+}$.In the present paper, we follow the techniques on integration estimates from[14] and establish the following local L¡Þ estimate of A in terms of the initial geometry and the |HA| bound along the flow

In our previous paper \cite{GWX}, a weighted Minkowski type formula (see \eqref{mink-zero} below) in a horoball in hyperbolic space has been proved by the author with Wang and Xia. Motivated by this formula and the paper of Wang-Weng \cite{WW}, it is natural to study a locally constrained Guan-Li type mean curvature flow with capillary boundary in a hyperbolic horoball, which is our purpose in the paper.

In the following we review some conceptions about horosphere and horoball in hyperbolic space. We use the upper half-space model for the hyperbolic space $\hh^{n+1}$, which is denoted by
\begin{eqnarray}\label{half-space}
\hh^{n+1}=\{x=(x_{1}, x_{2},\cdots,x_{n+1})\in \rr^{n+1}_+: x_{n+1}>0\},\quad \bar g=\frac{1}{x_{n+1}^2}\delta.
\end{eqnarray}
%\begin{definition}%(\cite{RAF})\hh^{n+1}=\{x=(x_{1}, x_{2},\cdots,x_{n+1})\in \left{\mathbb{R}^{n},\quad \bar g=\frac{1}{x_{n+1}^2}\delta\right}

A horosphere, a ``sphere" in $\hh^{n+1}$ whose centre lies at $\partial_{\infty}\mathbb{H}^{n+1}$, up to a hyperbolic isometry,  is given by the horizontal plane
\begin{equation}%\label{horo12}
\mathcal{H}:=\{x\in\mathbb{R}^{n+1}_{+}:x_{n+1}=1\}.
\end{equation}
%and a Euclidean sphere of $\mathbb{R}^{n}_{+}$ tangent to $\partial_{\infty}\mathbb{H}^{n}$. Therefore,
By choosing $\bar N=-E_{n+1}=(0,\cdots, 0, -1)$, all principal curvatures of a horosphere are $\kappa=1$. Moreover, by the Gauss equation, the induced metric on a horosphere is flat and in fact a horosphere is isometric to the $n$-dimensional Euclidean space.

In this model, a horoball is denoted by
\begin{equation}\label{horoball}
\mathcal{H}_{+}:=\{x\in\mathbb{R}^{n+1}_{+}:x_{n+1}>1\}.
\end{equation}
Different from geodesic ball, the horoball is a non-compact domain in $\hh^{n+1}$ whose boundary is $\partial\mathcal{H}_{+}=\mathcal{H}$. See Figure 1 below.
\begin{figure}[H]
	\centering
	\includegraphics[height=6.5cm,width=12cm]{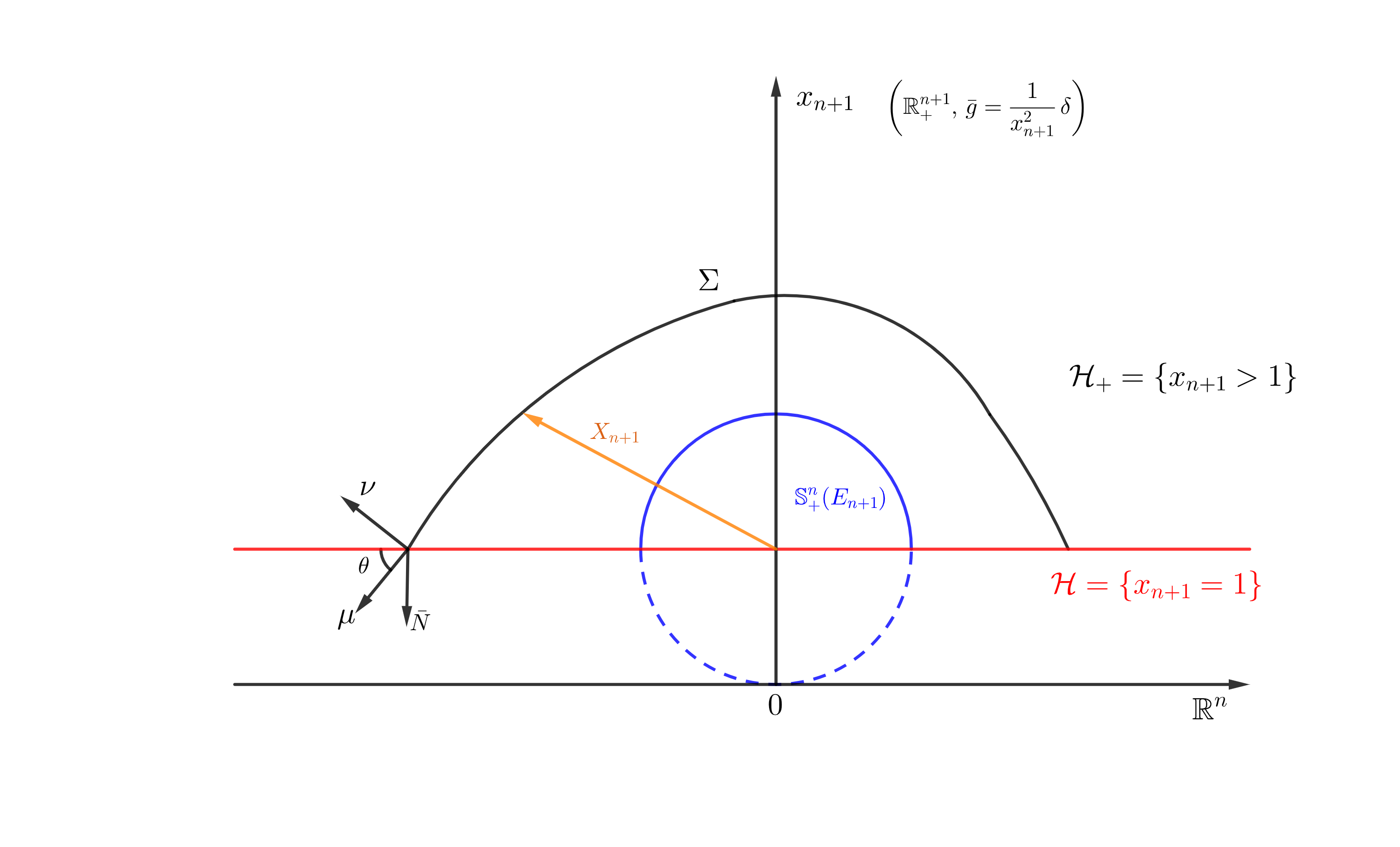}
	\caption{Hypersurface $\Sigma$ with $\theta$-capillary boundary in the horoball $\mathcal{H}_{+}$.}
\end{figure}
We use $x$ to denote the position vector in $\hh^{n+1}$ and $\bar{\nabla}$  the Levi-Civita connection of $\hh^{n+1}$. Let $\{E_i\}_{i=1}^{n+1}$ be the canonical basis of $\rr^{n+1}$.
We use $\<\cdot ,\cdot\>$ and $\bar g$ to denote the inner product of $\rr^{n+1}$ and $\hh^{n+1}$ respectively,  $D$ and $\bar \nabla$ to denote the Levi-Civita connection of $\rr^{n+1}$ and $\hh^{n+1}$ respectively.  Let $\bar E_i = x_{n+1} E_i$. Then $\{ \bar E_i\}_{i=1}^{n+1}$  is an orthonormal basis of $\hh^{n+1}$.
The relationship of $\bar{\nabla}$ and $D$ is given by
\begin{equation}\label{YZ}
  \bar{\nabla}_{Y}Z=D_{Y}Z-Y(\ln x_{n+1})Z-Z(\ln x_{n+1})Y+\langle Y,Z\rangle D(\ln x_{n+1}).
\end{equation}
It follows that
\begin{eqnarray}  \label{a1}
\bar \n_Y x &=& -\bar g (Y, \bar E_{n+1})x + \bar g (Y, x ) \bar E_{n+1}, \\ \label{a2}
\bar \n_Y  E_\a &=& -\bar g (Y, \bar E_{n+1}) E_\a + \bar g (Y, \bar E_\a) E_{n+1}, \quad \forall \a=1,2, \cdots, n, \\ \label{a3}
\bar \n_Y E_{n+1}  &=& - \frac 1{x_{n+1}} Y,
\end{eqnarray}
for any vector $Y$ in $\hh^{n+1}$.

%The following simple facts  play an important role in our paper.
\begin{prop}[{\rm \cite[Proposition 2.3]{GWX}}]\label{lem2.1} \

$(\rm i)$  The vector fields  \,$x$ and $\{E_\alpha\}_{\alpha=1}^{n}$ are Killing vector fields in $\hh^{n+1}$, i.e,
\begin{eqnarray}\label{xKilling}
\frac12\big( \bar g (\bar \n_i x, E_j) + \bar g (\bar \n_j x, E_i)
 ) =\frac12\big( \bar g( \bar \n_i E_{\alpha}, E_j)+
 \bar g( \bar \n_j E_{\alpha}, E_i)
 \big)=0.
\end{eqnarray}

$(\rm {ii})$\,$E_{n+1}$ is a conformal Killing vector field in $\hh^{n+1}$, i.e,
\begin{eqnarray}\label{En-Killing}
\frac12\big (\bar g( \bar \n_i E_{n+1}, E_j)+
\bar g( \bar \n_j E_{n+1}, E_i)
\big)= -\frac{1}{x_{n+1}}\bar{g}_{ij}.
\end{eqnarray}
Here $\bar \n_i = \bar \n _{E_i}$ and $\bar g_{ij}= \bar g(E_i, E_j)$.
\end{prop}

\begin{definition}\label{defi}
For $\theta\in(0,\pi)$, $\Sigma$ is said to be a hypersurface in $\mathcal{H}_{+}$ with $\theta$-capillary boundary if $\Sigma$ meets $\mathcal{H}$ at a constant angle $\theta$, that is,
\begin{equation}\label{angle}
  \bar{g}(\bar{N},\nu)=-\cos\theta \quad \text{along}\quad \partial\Sigma,
\end{equation}
where $\nu$ is the outward unit normal vector of $\Sigma$. In particular, if $\theta=\frac{\pi}{2}$, we call $\Sigma$ is free boundary hypersurface.
\end{definition}

\begin{definition}\label{defi-umbilical}
For any given constant $\theta \in(0, \pi)$, we define
\begin{equation}\label{qaaq}
\widehat{{C}_{\theta, r}(E_{n+1})}:=\{x\in\mathbb{R}^{n+1}_{+}:\ |x-(1-r\cos\theta)E_{n+1}|\leq r\,\,\text{and}\,\, x_{n+1} \geq1 \}
, \quad \forall r \in(0, \infty)
\end{equation}
and
\begin{equation}\label{model1}
C_{\theta, r}(E_{n+1}):=\{x\in\mathbb{R}^{n+1}_{+}:\ |x-(1-r\cos\theta)E_{n+1}|=r\,\,\text{and}\,\, x_{n+1} \geq1 \},\quad \forall r \in(0, \infty).
\end{equation}
\end{definition}
\begin{remark}
Here $C_{\theta, r}(E_{n+1})$ is a non-compact umbilical hypersurface around the vector field $E_{n+1}$ with principal curvature $\kappa=\frac{1}{r}-\cos\theta$, which is a model example of properly embedded hypersurfaces with $\theta$-capillary boundary, see \cite[Chapter 10]{Lop}.
\end{remark}

\begin{prop}\label{propsjjjsj}
Along $C_{\theta, r}(E_{n+1})$ we have
\begin{equation}\label{Cr}
  \left(\frac{1}{x_{n+1}}-\cos\theta \bar{g}(x,\nu)\right)=\kappa\,\bar{g}(x-E_{n+1}, \nu),
\end{equation}
where $\kappa$ is the principal curvature of $C_{\theta, r}(E_{n+1})$.
\end{prop}
\begin{proof}
From Definition \ref{defi-umbilical}, we know that the outer unit normal vector of $C_{\theta, r}(E_{n+1})$ is
\begin{equation}\label{normal-1}
  \nu=x_{n+1}\cdot\frac{x-(1-r\cos\theta)E_{n+1}}{r}.
\end{equation}
By using \eqref{normal-1} we have
\begin{eqnarray}
\left(\frac{1}{x_{n+1}}-\cos\theta \bar{g}(x,\nu)\right)&=&\frac{1}{x_{n+1}}-\cos\theta \bar{g}\left(\frac{r\nu}{x_{n+1}}+(1-r\cos\theta)E_{n+1},\nu\right)\\
&=&(1-r\cos\theta)\left(\frac{1}{x_{n+1}}-\cos\theta \bar{g}(E_{n+1},\nu)\right)\nonumber\\
&=&\left(\frac{1}{r}-\cos\theta\right)\bar{g}(x-E_{n+1}, \nu),\nonumber
\end{eqnarray}
\end{proof}

Next we introduce a conformal Killing vector field $X_{n+1}$ and a function $V_{n+1}$ in $\hh^{n+1}$ that  we will use
later. Denote
  \begin{equation}\label{cfkill2}
     X_{n+1}:=x-E_{n+1}, \quad V_{n+1}:=\frac{1}{x_{n+1}}.
  \end{equation}
\begin{prop}[{\rm \cite[Proposition 2.3]{GX2}}]\label{xaa}\

$(i)$\,$X_{n+1}$ is a conformal Killing vector field with $\frac{1}{2}\mathcal{L}_{X_{n+1}}\bar{g}=V_{n+1}\bar{g}$, namely
\begin{eqnarray}\label{XXaeq1}
\frac12\big[\bar \n_i (X_{n+1})_j+ \bar \n_j (X_{n+1})_i\big]=V_{n+1}\bar{g}_{ij}.
\end{eqnarray}
%\begin{itemize}

$ ({ii})$ $X _{n+1}\mid_{\mathcal{H}}$ is a tangential vector field on $\mathcal{H}$, i.e.,
\begin{eqnarray}\label{XXaeq2}
\bar{g}(X_{n+1}, \bar{N})=0 \quad\text{\rm on}\,\,\mathcal{H}.
\end{eqnarray}
%where $\bar{N}$ is unit outward normal vector on $B_{K,\lambda}^{\rm int}$.
%\end{itemize}[{\rm \cite[Proposition 3.2]{BCo}}]
\end{prop}
\begin{prop}[{\rm \cite[Proposition 2.4]{GX2}}]\label{xaa2} \, $V_{n+1}$ satisfies the following properties:
\
\begin{eqnarray}\label{va2}
  \bar{\n}^2 V_{n+1}&=& V_{n+1}  \bar{g } \quad\quad\text{in}\,\, \hh^{n+1},\label{Vn}\\
\partial_{\bar{N}}V_{n+1}&=& V_{n+1} \quad  \quad\,\,\text{on}\,\, \mathcal{H}.\label{cur}
\end{eqnarray}
\end{prop}
Utilizing the above properties of $X_{n+1}$, the author with Wang and Xia showed that a weighted Minkowski type formula with $\theta$-capillary boundary in a horoball $\mathcal{H}_{+}$ as follows \cite{GWX}
\begin{equation}\label{mink-zero}
 \int_{\Sigma} \left(\frac{n}{x_{n+1}}-n\cos \theta \bar{g}(x,\nu)\right)d A=\int_{\Sigma} H\bar{g}(X_{n+1}, \nu) dA,
\end{equation}
where $H$ is the mean curvature of $\Sigma$.

Motivated by the formula \eqref{mink-zero}, we consider a locally constraint Guan-Li type mean curvature flow with $\theta$-capillary boundary in $\mathcal{H}_{+}$. Let $\Sigma_{t}, \,t \in[0, T)$ be a family of hypersurfaces with boundary $\partial\Sigma_{t}$ in $\mathcal{H}_{+}$ given by a family of isometric embeddings $x(\cdot, t): M \rightarrow \bar{\mathcal{H}}_{+}$ from a compact manifold $M$ with boundary $\partial M$, such that
\begin{equation}\label{flowsss}
  \operatorname{int}\left(\Sigma_{t}\right)=x(\operatorname{int}(M), t) \subset \mathcal{H}_{+}, \quad \partial \Sigma_{t}=x(\partial M, t) \subset \mathcal{H} .
\end{equation}
Let $x(\cdot, t)$ satisfy
\begin{equation}\label{mainflow}
\begin{cases}{}
(\partial_{t} x)^{\perp}=f \nu & \text { in } M \times[0, T), \\
\bar{g}(\nu, \bar{N} \circ x)=-\cos\theta & \text { on } \partial M \times[0, T),\\
x(\cdot,0)=x_{0}& \text { in } M.
\end{cases}
\end{equation}
Here
\begin{equation}\label{ffff}
f:=\frac{n}{x_{n+1}}-n\cos\theta\bar{g}(x,\nu)-H \bar{g}\left(X_{n+1}, \nu\right),
\end{equation}
and $X_{n+1}$ is given by \eqref{cfkill2}.

%$\bar{N}=-x_{n+1}E_{n+1}$ be the unit outward normal of $\mathcal{H}$ in $\mathcal{H}_{+}$.
%The unit normal $\nu$ of $\Sigma$ is chosen in the following way. Let
%$\Omega_{t}$ be the component of the enclosed domain by $\Sigma_{t}$. The boundary condition $\bar{g}(\nu, \bar{N} \circ x)=0$ means $\Sigma_{t}$ intersects horosphere $\mathcal{H}_{+}$ orthogonally which is called $\Sigma_{t}$ with free boundary. The choice of speed function $F$ in (1.5) is clearly based on the Minkowski formula (1.4).

\begin{definition}\label{starshaped}
A hypersurface $\Sigma \subset \mathcal{H}_{+}$ is called star-shaped hypersurface with respect to $E_{n+1}$ if $\bar{g}\left(X_{n+1}, \nu\right)>0$ along $\Sigma$.
\end{definition}

Our main result in this paper is the following theorem.

\begin{theorem}\label{mainthm}
Let $x_{0}: M \rightarrow \mathcal{H}_{+}$ be a star-shaped hypersurface with $\theta$-capillary boundary in $\mathbb{H}^{n+1}$. Assume $|\cos\theta|<\frac{3n+1}{5n-1}$. Then the flow \eqref{mainflow} exists for $t \in[0, \infty)$ with uniform $C^{\infty}$-estimates. Moreover, $x(\cdot, t)$ converges in the $C^{\infty}$ topology as $t \rightarrow \infty$ to an umbilical hypersurface around the vector $E_{n+1}$ whose enclosed volume is the same as that of $x_{0}$.
\end{theorem}

\begin{remark}\

\begin{itemize}
 \item [(i)]The hypothesis of $|\cos\theta|<\frac{3n+1}{5n-1}$ is a technique reason as in \cite{MW,WW} when we prove the uniform gradient
estimate to the flow \eqref{mainflow}, see Section \ref{sec3} in detail. We hope this assumption could be removed in the future.
  \item [(ii)]The above umbilical hypersurface in $\hh^{n+1}$ could be a piece of a totally geodesic hyperplane, an equidistant hypersurface, a horosphere or a geodesic sphere.

\end{itemize}

\end{remark}
In particular, when the contact angle is $\pi/2$, we have
\begin{corollary}\label{freebdythm}
Let $x_{0}: M \rightarrow \mathcal{H}_{+}$ be a star-shaped hypersurface with free boundary in $\mathbb{H}^{n+1}$. Then the flow \eqref{mainflow} exists for $t \in[0, \infty)$ with uniform $C^{\infty}$-estimates. Moreover, $x(\cdot, t)$ converges in the $C^{\infty}$ topology as $t \rightarrow \infty$ to an umbilical hypersurface around $E_{n+1}$ whose enclosed volume is the same as that of $x_{0}$.
\end{corollary}

As an application, we get the following isoperimetric type inequalities for starshaped hypersurfaces with capillary boundary in a horoball, which can be viewed as the special case in our previous result \cite{GWX}.
\begin{corollary}\label{corollay222} Among the star-shaped hypersurfaces with capillary boundary such that fixed volume of the enclosed domain in a horoball, umbilical hypersurfaces are the energy minimizers provided the contact angle $\theta$ satisfying $|\cos\theta|<\frac{3n+1}{5n-1}$.
\end{corollary}
In fact, from \cite[Proposition 4.]{CP}, it holds
\begin{equation}\label{ddsdsd}
  \int_{\Sigma}\left(\frac{1}{x_{n+1}}-\cos\theta \bar{g}(x,\nu)\right)H\,dA=\frac{2}{n-1}\int_{\Sigma}\sigma_{2}\bar{g}(X_{n+1},\nu)dA,
\end{equation}
where $\sigma_{2}(\kappa)$ is the second order mean curvature. From the first variation of the enclosed volume functional $\operatorname{Vol}\left(\Omega_{t}\right)$ and the energy functional $\mathcal{E}(t):={Area}(\Sigma_{t})-\cos\theta\, Wet(\Sigma_{t})$, see e.g \cite{RS, GWX}, we have
\begin{equation}\label{zasqq}
\frac{d}{d t} \operatorname{Vol}\left(\Omega_{t}\right)=\int_{\Sigma_{t}} f \,d A_{t}=0
\end{equation}
and
\begin{eqnarray}\label{energy-1}
\frac{d}{d t} \mathcal{E}(t)&=&\int_{\Sigma_{t}} \left(\frac{n}{x_{n+1}}-n\cos\theta \bar{g}(x,\nu)-H\bar{g}(X_{n+1},\nu) \right)H \,d A_{t} \\
&=&-\frac{1}{n-1} \int_{\Sigma_{t}} \sum_{1 \leq i<j \leq n}\left(\kappa_{i}-\kappa_{j}\right)^{2} \bar{g}\left(X_{n+1}, \nu\right) d A_{t}\leq0, \nonumber
\end{eqnarray}
where $\{\kappa_{i}\}_{i=1}^{n}$ are the principal curvatures of $\Sigma_{t}$. That is, along the flow \eqref{mainflow}, the enclosed domain between $\Sigma$ and horosphere $\mathcal{H}$ has fixed volume while its energy is monotone decreasing. Hence Corollary \ref{corollay222} follows directly from Theorem \ref{mainthm}.

\

This paper is organized as follows: In Section \ref{sec2} we reduce the Guan-Li type flow \eqref{mainflow} to scalar parabolic equation over the umbilical hypersurface $\mathbb{S}^{n}_{+}(E_{n+1})$ in a horoball in the hyperbolic space. In Section \ref{sec3} we get the uniform a priori estimates of flow \eqref{mainflow} and prove Theorem \ref{mainthm}.
\section{Scalar flow }\label{sec2}
In this section, we will reduce the flow \eqref{mainflow} to a scalar flow if the initial hypersurface is star-shaped. Under the upper half-space model \eqref{half-space}, we let
\begin{equation*}\label{SEn}
  \mathbb{S}^{n}_{+}(E_{n+1}):=\{x\in\mathbb{R}_{+}^{n+1}\mid|x-E_{n+1}|=1\,\, \text{and}\,\, x_{n+1}>1\}
\end{equation*}
and $\mathbb{S}^{n-1}(E_{n+1}):=\partial\mathbb{S}^{n}_{+}(E_{n+1})$.
Since $\mathcal{H}_{+}\subset\mathbb{R}_{+}^{n+1}$, we use the polar coordinate $(\rho, \beta, \xi) \in[0,+\infty) \times\left[0, \frac{\pi}{2}\right] \times \mathbb{S}^{n-1}(E_{n+1})$, where $\xi$ is the spherical coordinate on $\mathbb{S}^{n-1}(E_{n+1})$ and
\begin{equation}\label{coord}
\begin{cases}{}
\rho^{2}=|x|^{2}+(x_{n+1}-1)^{2}, \\
x_{n+1}=\rho\cos\beta+1,\,\,|x|=\rho\sin\beta.
\end{cases}
\end{equation}
Thus it implies that the standard Euclidean metric has the following expression
\begin{equation}\label{metric}
  \delta_{\mathbb{R}^{n+1}_{+}}=d\rho^{2}+\rho^{2}g_{\mathbb{S}^{n}_{+}(E_{n+1})}=d\rho^{2}+\rho^{2}d\beta^{2}+\rho^{2}\sin^{2}\beta g_{\mathbb{S}^{n}_{+}(E_{n+1})},
\end{equation}
where $g_{\mathbb{S}^{n}_{+}(E_{n+1})}$ is the standard spherical metric on $\mathbb{S}^{n}_{+}(E_{n+1})$.

A hypersurface $\Sigma \subset\left(\mathcal{H}_{+}, \bar{g}\right)$ is star-shaped with respect to $E_{n+1}$, i.e., $\bar{g}(X_{n+1}, \nu)>0$ on $\Sigma$.  We can view $\bar{\Sigma}:=\overline{x(M)}$ as a radial graph over $\bar{\mathbb{S}}_{+}^{n}(E_{n+1})$. Therefore, $\Sigma$ can be written by
\begin{equation}\label{graph}
 X_{n+1}:=x-E_{n+1}=\rho(z)z=\rho(\beta,\xi)z,\quad z:=(\beta,\xi)\in\mathbb{S}^{n}_{+}(E_{n+1}).
\end{equation}
In polar coordinate, a direct computation implies that
\begin{equation}\label{Ennn}
  \frac{\partial}{\partial x_{n+1}}=\frac{\partial\rho}{\partial x_{n+1}}\partial_{\rho}+\frac{\partial\beta}{\partial x_{n+1}}\partial_{\beta}=\cos\beta\partial_{\rho}-\frac{\sin\beta}{\rho}\partial_{\beta}.
\end{equation}
 Set $\omega:=-\log x_{n+1}=-\log(\rho\cos\beta+1)$ and $u:=\log\rho$ and $v:=\sqrt{1+|\nabla u|^{2}}$. It is well-known that ${\nu}_{\delta}=\frac{\partial_{\rho}-\rho^{-1} \nabla u}{v}$ is the unit outward normal vector of ${\Sigma}$ in $\left(\mathbb{R}_{+}^{n+1}, \delta_{\mathbb{R}_{+}^{n+1}}\right).$ Then the capillary boundary condition in \eqref{mainflow} gives us that
\begin{equation}\label{boundary1}
-\cos\theta=\bar{g}(\nu,\bar{N}\circ x)=-\langle{\nu}_{\delta},E_{n+1}\rangle=\left\langle\frac{\partial_{\rho}-\rho^{-1} \nabla u}{v},\,\frac{1}{\rho}\partial_{\beta}\right\rangle=-\frac{1}{v}\nabla_{\beta}u.
\end{equation}
It follows that
\begin{equation}\label{bd2}
  \nabla_{\partial_{\beta}} u=\cos \theta v \quad \text { on } \partial{\mathbb{S}}_{+}^{n}(E_{n+1}).
\end{equation}
By a straightforward computation as above, we have
\begin{eqnarray}\label{bd3}
\qquad\bar{g}(X_{n+1},\nu)=\bar{g}(x-E_{n+1},\nu)=\frac{1}{x_{n+1}}\left\langle\rho\partial_{\rho},\nu_{\delta}\right\rangle=\frac{1}{x_{n+1}}\left\langle\rho\partial_{\rho},\frac{\partial_{\rho}-\rho^{-1} \nabla u}{v}\right\rangle=\frac{1}{v}\rho e^{\omega}
\end{eqnarray}
and
\begin{eqnarray}\label{bd4}
\bar{g}(E_{n+1},\nu)&=&\frac{1}{x_{n+1}}\left\langle E_{n+1},\nu_{\delta}\right\rangle\\
&=&\frac{1}{x_{n+1}}\left\langle\cos\beta\partial_{\rho}-\frac{\sin\beta}{\rho}\partial_{\beta},\frac{\partial_{\rho}-\rho^{-1} \nabla u}{v}\right\rangle\nonumber\\
&=&\frac{1}{v}e^{\omega}(\cos\beta+\sin\beta\nabla_{\beta}u).\nonumber
\end{eqnarray}
Combining \eqref{bd3} and \eqref{bd4}, we get
\begin{eqnarray}\label{bd5}
\bar{g}(x,\nu)=\bar{g}(X_{n+1},\nu)+\bar{g}(E_{n+1},\nu)=\frac{1}{v}e^{\omega}(\rho+\cos\beta+\sin\beta\nabla_{\beta}u).
\end{eqnarray}
Recall that $e^{-\omega}=x_{n+1}=\rho\cos\beta+1$. Then it follows that
\begin{eqnarray}\label{bd6}
D_{\nu_{\delta}}e^{-\omega}&=&\langle De^{-\omega},\nu_{\delta}\rangle=\langle D_{\rho}e^{-\omega}\partial_{\rho}+(D_{\rho^{-1}\partial_{\beta}}e^{-\omega})(\rho^{-1}\partial_{\beta}),\nu_{\delta}\rangle\\
&=&\left\langle \cos\beta\partial_{\rho}-\rho^{-1}\sin\beta\partial_{\beta},\frac{\partial_{\rho}-\rho^{-1}\nabla u}{v}\right\rangle\nonumber\\
&=&\frac{1}{v}(\cos\beta+\sin\beta\nabla_{\beta}u).\nonumber
\end{eqnarray}

Let $H_{\delta}$ be a mean curvature with respect to ${\nu}_{\delta}$ of ${\Sigma}$ in $\left(\mathbb{R}_{+}^{n+1}, \delta_{\mathbb{R}_{+}^{n+1}}\right)$.  Then $H$ and $H_{\delta}$ have the following relations
\begin{eqnarray}
{H} &=&e^{-\omega}\left(H_{\delta}+n D_{{\nu}_{\delta}} \omega\right) \\
&=&e^{-\omega}\left(\frac{n}{\rho v}-\frac{1}{\rho v} \sum_{i, j=1}^{n}\left(\sigma^{i j}-\frac{u^{i} u^{j}}{v^{2}}\right) u_{i j}\right)-n D_{\nu_{\delta}} e^{-\omega}\nonumber \\
&=&e^{-\omega}\left(\frac{n}{\rho v}-\frac{1}{\rho v} \sum_{i, j=1}^{n}\left(\sigma^{i j}-\frac{u^{i} u^{j}}{v^{2}}\right) u_{i j}\right)-\frac{n}{v}(\cos \beta+\sin \beta \nabla_{\beta}u) \nonumber\\
&=&-\frac{1}{\rho v e^{\omega}} \sum_{i, j=1}^{n}\left(\sigma^{i j}-\frac{u^{i} u^{j}}{v^{2}}\right) u_{i j}+\frac{n}{\rho v}-\frac{n\sin\beta}{v}\nabla_{\beta}u.\nonumber
\end{eqnarray}
 We know that the first equation in flow \eqref{mainflow} is reduced to the following scalar equation
\begin{equation}\label{scal}
  \frac{\partial\rho}{\partial t}=\frac{v}{e^{w}}f,
\end{equation}
where
\begin{eqnarray}
f&=&\frac{n}{x_{n+1}}-n\cos\theta\bar{g}(x,\nu)-H \bar{g}\left(X_{n+1}, \nu\right)\\
&=&ne^{\omega}-\frac{n\cos\theta}{v}e^{\omega}(\rho+\cos\beta+\sin\beta\nabla_{\beta}u)\nonumber\\
&&+\frac{\rho e^{\omega}}{v}\left(\frac{1}{e^{\omega}\rho v} \sum_{i, j=1}^{n}\left(\sigma^{i j}-\frac{u^{i} u^{j}}{v^{2}}\right) u_{i j}-\frac{n}{\rho v}+\frac{n\sin\beta}{v}\nabla_{\beta}u
\right).\nonumber
\end{eqnarray}
Note that $u=\log\rho$ and $v^{2}=1+|\nabla u|^{2}$. Hence \eqref{scal} is also equivalent to
\begin{eqnarray}\label{inter-flow}
\frac{\partial u}{\partial t}&=&\frac{v}{\rho e^{\omega}}f\\
&=&\frac{1}{\rho e^{\omega}v} \sum_{i, j=1}^{n}\left(\sigma^{i j}-\frac{u^{i} u^{j}}{v^{2}}\right) u_{i j}+\frac{n}{\rho v}|\nabla u|^{2}+\frac{n\sin\beta}{v}\nabla_{\beta}u\nonumber\\
&&-\frac{n\cos\theta}{\rho}(\rho+\cos\beta+\sin\beta\nabla_{\beta}u)\nonumber\\
&:=&G(\nabla^{2}u,\nabla u,\rho,\beta).\nonumber
\end{eqnarray}
In summary, combining \eqref{inter-flow} and \eqref{boundary1}, the flow $\eqref{mainflow}$ is equivalent to the following scalar parabolic equation on $\mathbb{S}^{n}_{+}(E_{n+1})$
\begin{equation}\label{scalflow}
\begin{cases}{}
u_{t}=G(\nabla^{2}u,\nabla u,\rho,\beta)& \text{in} \,\,\,\mathbb{S}^{n}_{+}(E_{n+1})\times[0,T), \\
\nabla_{\beta}u=\cos\theta\sqrt{1+|\nabla u|^{2}} &\text{on}\,\,\,\partial\mathbb{S}^{n}_{+}(E_{n+1})\times[0,T),\\
u(\cdot,0)=u_{0}&\text{on}\,\,\,\mathbb{S}^{n}_{+}(E_{n+1}),
\end{cases}
\end{equation}
where $u_{0}:=\log\rho_{0}$, $\rho_{0}$ is only related to the initial hypersurface $x_{0}(M)$ and $G$ is defined by \eqref{inter-flow}.

\section{A priori estimates and convergence}\label{sec3}
The short time existence of the flow $\eqref{mainflow}$ is established by the standard PDE theory, due to our assumption of star-shaped (i.e. $\bar{g}(X_{n+1}, \nu)>0$) for the initial hypersurface $x_{0}(M)$, this flow is transformed into the scalar flow \eqref{scalflow}, which is uniformly parabolic.

In the following we will show the uniform height and gradient estimates for the equation \eqref{scalflow}, then the uniform $C^{\infty}$ estimates and the long time existence of solution to flow follow from the standard parabolic PDE theory. Now we prove that the radial function $u$ has $C^0$ estimate.
\begin{prop}\label{prop3.1}
Assume that the initial star-shaped hypersurface $x_0(M)\subseteq\mathcal{H}_{+}$ satisfies
\begin{equation*}
  x_0(M) \widehat{\subset {C}_{\theta, R_2}(E_{n+1})} \backslash \widehat{{C}_{\theta, R_1}(E_{n+1})},
\end{equation*}
for some $R_2>R_1>0$, where $\widehat{{C}_{\theta, r}(E_{n+1})}$ is defined by \eqref{qaaq}. Then this property is preserved along flow \eqref{mainflow}. In particular, if $u(x, t)$ solves \eqref{scalflow}, then
\begin{equation}\label{kiisq}
  \|u\|_{C^0\left(\mathbb{S}_{+}^n(E_{n+1}) \times[0, T)\right)} \leq C,
\end{equation}
where $C$ is a constant only depending on the initial value.
\end{prop}
\begin{proof}
From Proposition \ref{propsjjjsj}, we know that the umbilical hypersurface $C_{\theta,r}(E_{n+1})$ satisfies
\begin{equation}\label{static-1}
 \left(\frac{n}{x_{n+1}}-n\cos\theta \bar{g}(x,\nu)\right)-H\bar{g}(x-E_{n+1}, \nu)=0,
\end{equation}
that is, it is a static solution to flow \eqref{mainflow} for each $r>0$. Thus the assertion follows from the avoidance principle for the strictly parabolic equation with $\theta$-capillary boundary condition (see e.g \cite[Proposition 4.2]{WW}).
\end{proof}
For the convenience of notation, we use $\sigma$ to denote the metric of $\mathbb{S}_{+}^n(E_{n+1})$, i.e. ${g}_{\mathbb{S}_{+}^n(E_{n+1})}$ and we write $O(s)$ to be the terms that are bounded by $C s$ for some positive constant $C$, which depends only on the $C^0$ norm of $u$.

We let the distance function $d(x):=$ $\operatorname{dist}\left(x, \partial \mathbb{S}_{+}^n(E_{n+1})\right)$ on $\mathbb{S}_{+}^n(E_{n+1})$ for the metric $\sigma$. Thus $d$ is smooth well-defined function for $x$ near $\partial \mathbb{S}_{+}^n(E_{n+1})$ and $\nabla d=-\partial_\beta$ on $\partial \mathbb{S}_{+}^n(E_{n+1})$, where $\partial_\beta$ is the outward unit normal vector of $\partial \mathbb{S}_{+}^n(E_{n+1})$ in $\mathbb{S}_{+}^n(E_{n+1})$. Extending $d$ to be a smooth function defined in the compact hypersurface $\overline{\mathbb{S}}_{+}^n(E_{n+1})$, then it satisfies
\begin{equation}\label{ddddd}
  d \geq 0, \quad|\nabla d| \leq 1, \quad|\nabla^{2} d| \leq C, \quad\text { in } \overline{\mathbb{S}}_{+}^n(E_{n+1}).
\end{equation}

In order to prove the $C^{1}$ estimate, we next choose an auxiliary function $\Phi$ that had been used in \cite[Proposition 4.3]{WW} to obtain the uniform gradient estimate for the flow \eqref{scalflow}.

\begin{prop}\label{proposition3.2}
Assume $u: \mathbb{S}_{+}^{n}(E_{n+1}) \times[0, T) \rightarrow \mathbb{R}$ solves \eqref{scalflow} and $|\cos\theta|<\frac{3n+1}{5n-1}$, then
\begin{equation}\label{C1}
  |\nabla u|(x, t) \leq C,
\end{equation}
for any $(x, t) \in \mathbb{S}_{+}^{n}(E_{n+1}) \times[0, T)$, where $C$ is a positive constant depends only on the initial value.
\end{prop}
\begin{proof}Define the function
\begin{equation}\label{pjs}
  \Phi:=(1+Kd) v+\cos \theta \sigma(\nabla u, \nabla d),
\end{equation}
where $K$ is a positive constant to be determined later. For any $T^{\prime}<T$, assume $\Phi$ attains its maximum value at some point $\left(x_0, t_0\right) \in \bar{\mathbb{S}}_{+}^n(E_{n+1}) \times\left[0, T^{\prime}\right]$.

Following the same ideas in \cite[Proposition 4.3, Case 1]{WW}, by choosing $K>0$ sufficiently large, $\Phi$ does not attach its maximum value on $\partial \mathbb{S}_{+}^n(E_{n+1})$, hence we have either

$(i)$ \,\,$(x_{0},t_{0})\in\mathbb{S}_{+}^n(E_{n+1})\times\{0\}$ or

$(ii)$ $(x_{0},t_{0})\in\mathbb{S}_{+}^n(E_{n+1})\times(0,T']$.

When case $(i)$ happen, we can see that
\begin{equation*}
  \sup _{\overline{\mathbb{S}}_{+}^n(E_{n+1}) \times\left[0, T^{\prime}\right]}|\nabla u| \leq C,
\end{equation*}
where $C$ is a positive constant depending only on $n$ and $u_0$.

When case $(ii)$ happen, we choose the geodesic coordinate $\left\{\frac{\partial}{\partial x_i}\right\}_{i=1}^n$ at $x_0$ and assume
$|\nabla u|=u_1>0$, $\left\{u_{\alpha \beta}\right\}_{2 \leq \alpha, \beta \leq n}$ is diagonal. Next all the computations are deduced at the point $\left(x_0, t_0\right)$. From \eqref{pjs}
$$
0=\nabla_i \Phi=(1+K d) v_i+K d_i v+\cos \theta\left(u_{l i} d_l+u_l d_{l i}\right), \quad\text{for any}\,\, \, 1 \leq i \leq n,
$$
it follows that
\begin{equation}
  \left[(1+K d) \frac{u_1}{v}+\cos \theta d_1\right] u_{11}=-\cos \theta u_{\alpha 1} d_\alpha-\cos \theta u_1 d_{11}-K d_1 v,
\end{equation}
and
\begin{equation}
  \left[(1+Kd) \frac{u_1}{v}+\cos \theta d_1\right] u_{1 \alpha}=-\cos \theta u_{\alpha \alpha} d_\alpha-\cos \theta u_1 d_{1 \alpha}-Kd_\alpha v,\quad \text{for any}\,\, \, 2\leq\alpha\leq n.
\end{equation}

Suppose $u_{1}(x_{0},t_{0}) \geq \delta_{0}$ for some constant $\delta_{0}>0$, otherwise, we have done. For the convenience, we denote $S:=(1+Kd) \frac{u_1}{v}+\cos \theta d_1$, then $0<C(\delta_{0}, \theta) \leq S \leq 2+K$ by \eqref{ddddd}. Hence
\begin{eqnarray}\label{hsuhaushaq2}
u_{11} &=&-\frac{1}{S} \cos \theta u_{\alpha 1} d_\alpha-\frac{1}{S}\left(\cos \theta u_1 d_{11}+K d_1 v\right) \\
&=&\frac{\cos ^2 \theta}{S^2} \sum_{\alpha=2}^n d_\alpha^2 u_{\alpha \alpha}+O(v),\nonumber
\end{eqnarray}
and
\begin{eqnarray}\label{hsuhaushaq}
u_{1 \alpha} & =&-\frac{\cos \theta d_\alpha}{S} u_{\alpha \alpha}-\frac{1}{S}\left(\cos \theta u_1 d_{1 \alpha}+K d_\alpha v\right) \\
& =&-\frac{\cos \theta d_\alpha}{S} u_{\alpha \alpha}+O(v),\quad \text{for any}\,\,\, 2 \leq \alpha \leq n.\nonumber
\end{eqnarray}
Set $G:=G(\nabla^{2} u, \nabla u=p, \rho, \beta)$ and $u_{\beta}:=\sigma(\nabla u,\partial_{\beta})$. By a direct computation we see
\begin{eqnarray}
&&\,\, G^{i j}:=\frac{\partial G}{\partial u_{i j}} =\frac{1}{\rho ve^{\omega}}\left(\sigma^{i j}-\frac{u^{i} u^{j}}{v^{2}}\right),\label{Gij}\\
&&\,\,G_{p_{i}}:=\frac{\partial G}{\partial u_{i}}=-\frac{u_{k}}{\rho e^{\omega}v^{3}}\sum_{i,j}\sigma^{ij}u_{ij}+\frac{3u_{k}}{\rho e^{\omega}v^{5}}\sum_{i,j}u_{ij}u_{i}u_{j}-\frac{2}{\rho e^{\omega}v^{3}}\sum_{i,j}u_{ik}u_{i}\\
&&\qquad\quad\quad+\left(\frac{2n}{\rho v}-\frac{n|\nabla u|^{2}}{\rho v^{3}}-\frac{n\sin\beta \nabla_{\beta}u}{v^{3}}\right)u_{k}+n\sin\beta\sigma(\partial_{k},\partial_{\beta})\left(\frac{1}{v}-\frac{\cos\theta}{\rho}\right),\nonumber\\
&&\,\,G_{\rho}:=\frac{\partial G}{\partial \rho}=-\frac{1}{\rho^{2}v}\sum_{i,j}\left(\sigma^{i j}-\frac{u^{i} u^{j}}{v^{2}}\right)u_{ij}-\frac{n}{\rho^{2}v}|\nabla u|^{2}+\frac{n\cos\theta}{\rho^{2}}(\cos\beta+\sin\beta\nabla_{\beta}u),\label{GFrou}\\
&&\,\,G_{\beta}:=\frac{\partial G}{\partial \beta}=-\frac{\sin\beta}{v}\sum_{i,j}\left(\sigma^{i j}-\frac{u^{i} u^{j}}{v^{2}}\right)u_{ij}+\frac{n\cos\beta}{v}\nabla_{\beta}u+\frac{n\cos\theta}{\rho}(\sin\beta-\cos\beta\nabla_{\beta}u).\label{GFb}
 \end{eqnarray}
By differentiating the first equation in \eqref{scalflow}, we get
\begin{equation}\label{utl}
  u_{t l}=G^{i j} u_{i j l}+G_{p_{i}} u_{i l}+G_\rho \rho u_l+G_\beta \sigma\left(\partial_\beta, \partial_l\right).
\end{equation}
Using the Ricci identity on $\mathbb{S}_{+}^n(E_{n+1})$, we have
\begin{equation}\label{sssxsa}
u_{i j l}=u_{l i j}+u_j \sigma_{l i}-u_l \sigma_{i j}.
\end{equation}
It follows that
\begin{equation}\label{sxsq}
u_{t l}=G^{i j} u_{l i j}+G^{i j} u_j \sigma_{li}-\sum_{i=1}^n G^{i i} u_l+G_{p_i} u_{i l}+\rho G_\rho u_{l}+G_{\beta} \sigma\left(\partial_\beta, \partial_l\right) .
\end{equation}
Now we define a linearized operator as
\begin{equation}\label{sxxzzaa}
\mathcal{L}:=\partial_{t}-G^{i j} \nabla_{i j}-G_{p_i} \nabla_i.
\end{equation}
Hence at $(x_{0},t_{0})$ we have
\begin{eqnarray}
\qquad\,\,\,\,0 \leq \mathcal{L}\Phi &= &\frac{(1+Kd)}{v} u_l\left(u_{l t}-G^{i j} u_{i j j}-G_{p_i} u_{l i}\right)+d_k \cos \theta\left(u_{k t}-G^{i j} u_{k i j}-G_{p_i} u_{k i}\right) \label{summary888}\\
&&+(1+K d)\left(\frac{G^{i j} u_l u_{l i} u_k u_{k j}}{v^3}-\frac{G^{i j} u_{l i} u_{l j}}{v}\right)-\left(2 \cos \theta G^{i j} u_{k i} d_{k j} +2K G^{i j} d_i v_j\right) \nonumber\\
&& -\left(\cos \theta G^{i j} u_k d_{k i j}+K v G^{i j} d_{i j}\right)-G_{p_i}\left(K v d_i +\cos \theta u_k d_{k i}\right) \nonumber\\
&:= & I_1+I_2+I_3+I_4+I_5+I_6 .\nonumber
\end{eqnarray}
In the following we will calculate above six terms $I_{1}-I_{6}$ one by one.

Firstly, we deal with the term $I_{1}$. By \eqref{sxsq} we have
\begin{eqnarray}\label{equ-nonc-111111ww}
\quad\quad\,\, I_1&=& \frac{(1+K d)}{v} u_l\left(u_{l t}-G^{i j} u_{l i j}-G_{p_i} u_{i l}\right) \\
&=& \frac{(1+Kd)}{v} G^{i j} u_j u_l \sigma_{l i}+\frac{(1+K d)}{v}\left(\rho G_\rho |\nabla u|^{2}+G_\beta u_{\beta}\right)-\frac{(1+Kd)|\nabla u|^2}{v} \sum_{i=1}^n G^{i i}. \nonumber
\end{eqnarray}
Inserting \eqref{Gij}-\eqref{GFb} into \eqref{equ-nonc-111111ww} to obtain
\begin{eqnarray}\label{equ-nonc-1111}
\quad\quad\,\, I_1
&=&\left[-\frac{(1+Kd)}{\rho v^{4}}|\nabla u|^{2}u_{11}-\frac{(1+Kd)}{v^{4}}\sin\beta u_{\beta}u_{11}\right]-\frac{(1+Kd)}{\rho v^{2}}|\nabla u|^{2}\sum_{\alpha=2}^{n}u_{\alpha\alpha} \nonumber\\
&&+\left[\frac{n(1+Kd)}{\rho v}|\nabla u|^{2}\cos\theta\sin\beta u_{\beta}-\frac{n(1+Kd)}{\rho v^{2}}|\nabla u|^{2}u_{1}^{2}\right] \nonumber\\
&&+\left[\frac{(1-n)(1+Kd)}{\rho v^{2}e^{\omega}}|\nabla u|^{2}+\frac{n(1+Kd)}{\rho v}\cos\theta(\sin\beta-\cos\beta u_{\beta})u_{\beta}\right. \nonumber\\
&&\left.-\frac{(1+Kd)}{v^{2}}\sin\beta u_{\beta}\sum_{\alpha=2}^{n}u_{\alpha\alpha}+\frac{n(1+Kd)}{v^{2}}\cos\beta u^{2}_{\beta}+\frac{n(1+Kd)}{\rho v}|\nabla u|^{2}\cos\theta\cos\beta \nonumber
\right]\nonumber\\
&:=&I_{11}+I_{12}+I_{13}+I_{14}, \nonumber
\end{eqnarray}
%where we used fact that $\sum\limits_{i,j}\sigma^{ij}u_{ij}=u_{11}+\sum\limits_{\alpha=2}^{n}u_{\alpha\alpha}$ due to $\left\{u_{\alpha \beta}\right\}_{2 \leq \alpha, \beta \leq n}$ is diagonal.
From \eqref{hsuhaushaq2}, we obtain
\begin{eqnarray}\label{equ-nonc-1qq}
I_{11}&=&-\frac{(1+Kd)}{\rho v^{4}}|\nabla u|^{2}u_{11}-\frac{(1+Kd)}{v^{4}}\sin\beta u_{\beta}u_{11}\\
&=&-\left(\frac{(1+Kd)}{\rho v^{4}}|\nabla u|^{2}+\frac{(1+Kd)}{v^{4}}\sin\beta u_{\beta}\right)\left(\frac{\cos^{2}\theta}{S^{2}}\sum_{\alpha=2}^{n}d^{2}_{\alpha}u_{\alpha\alpha}+O(v)\right)\nonumber\\
&\leq&O(\frac{1}{v^{2}})\sum_{\alpha=2}^{n}|u_{\alpha\alpha}|+O(\frac{1}{v})\nonumber
\end{eqnarray}
and
\begin{equation}\label{asaas}
  I_{14}=O(\frac{1}{v})\sum_{\alpha=2}^{n}|u_{\alpha\alpha}|+O(v).
\end{equation}
Similar to $I_{1}$, we can also calculate the term $I_2$ by \eqref{sxsq} and \eqref{Gij}-\eqref{GFb} to get
\begin{eqnarray}\label{equ-nonc-1qqaa}
\quad I_{2}&=&d_k \cos \theta\left(u_{k t}-G^{i j} u_{k i j}-G_{p_i} u_{k i}\right) \\
&=&\cos\theta\sigma(\nabla u, \nabla d) \rho G_\rho+\cos \theta G_\beta d_\beta+\cos\theta G^{ij}u_{j}\sigma_{ik}d_{k}-\cos\theta\sigma(\nabla u, \nabla d)\sum_{i=1}^{n}G^{ii}\nonumber\\
&=&-\left[\cos\theta\sigma(\nabla u, \nabla d)\frac{1}{\rho v^{3}}+\cos\theta\frac{\sin\beta}{v^{3}}d_{\beta}\right]u_{11}-\frac{1}{\rho v}\cos\theta\sigma(\nabla u, \nabla d)\sum_{\alpha=2}^{n}u_{\alpha\alpha}\nonumber\\
&&+\cos\theta\sigma(\nabla u,\nabla d)\left(\frac{n\cos\theta\sin\beta}{\rho} u_{\beta}-\frac{n}{\rho v}|\nabla u|^{2}\right)\nonumber\\
&&+\left[-\cos\theta d_{\beta}\frac{\sin\beta}{v}\sum_{\alpha=2}^{n}u_{\alpha\alpha}+\cos\theta\sigma(\nabla u, \nabla d)\frac{(1-n)}{\rho ve^{\omega}}\right.\nonumber\\
&&\quad+\left.\cos\theta d_{\beta}\frac{n\cos\beta}{v}u_{\beta}+\frac{n\cos^{2}\theta}{\rho}\left(\sin\beta d_{\beta}-\cos\beta u_{\beta}d_{\beta}+\cos\beta\sigma(\nabla u,\nabla d)\right)\right]\nonumber\\
&:=&I_{21}+I_{22}+I_{23}+I_{24}.\nonumber
\end{eqnarray}
For the term $I_{21}$ and $I_{24}$, by using \eqref{hsuhaushaq2} we see
\begin{eqnarray}\label{equ-nonc-1qqaa31}
I_{21}&=&-\left[\cos\theta\sigma(\nabla u,\nabla d)\frac{1}{\rho v^{3}}+\cos\theta\frac{\sin\beta}{v^{3} }d_{\beta}\right]\left(\frac{\cos^{2}\theta}{S^{2}}\sum_{\alpha=2}^{n}d_{\alpha}u_{\alpha\alpha}+O(v)\right)\\
&=&O(\frac{1}{v^{2}})\sum_{\alpha=2}^{n}|u_{\alpha\alpha}|+O(\frac{1}{v})\nonumber
\end{eqnarray}
and
\begin{equation}\label{sasisiqjsqx}
I_{24}=O(\frac{1}{v})\sum_{\alpha=2}^{n}|u_{\alpha\alpha}|+O(v).
\end{equation}
For the term $I_{3}$, we have
\begin{eqnarray*}
I_{3}&=&(1+Kd)\left(\frac{G^{ij}u_{l}u_{li}u_{k}u_{kj}}{v^{3}}-\frac{G^{ij}u_{li}u_{lj}}{v}\right)\nonumber\\
&=&-\frac{(1+Kd)}{\rho v^{4}e^{\omega}}\left(\frac{u^{2}_{11}}{v^{2}}+2\sum_{\alpha=2}^{n}u^{2}_{1\alpha}\right)-\frac{(1+Kd)}{\rho v^{2}e^{\omega}}\sum_{\alpha=2}^{n}u^{2}_{\alpha\alpha}\nonumber\\
&=&-\frac{(1+Kd)}{\rho v^{4}e^{\omega}}\left(\frac{u^{2}_{11}}{v^{2}}+2\sum_{\alpha=2}^{n}u^{2}_{1\alpha}\right)-(1-\varepsilon)\frac{(1+Kd)}{\rho v^{2}e^{\omega}}\sum_{\alpha=2}^{n}u^{2}_{\alpha\alpha}-\varepsilon\frac{(1+Kd)}{\rho v^{2}e^{\omega}}\sum_{\alpha=2}^{n}u^{2}_{\alpha\alpha}\nonumber\\
&:=&I_{31}+I_{32}+I_{33},\nonumber
\end{eqnarray*}
where $\varepsilon\in(0,1)$ is a constant to be determined later.

By \eqref{hsuhaushaq2} and \eqref{hsuhaushaq}, we derive 
\begin{eqnarray*}\label{equ-nonc-asaxqqq}
I_{31}&=&-\frac{(1+Kd)}{\rho e^{\omega}}\left[\frac{1}{v^{6}}\left(\frac{\cos^{2}\theta}{S^{2}}\sum_{\alpha=2}^{n}d^{2}_{\alpha}u_{\alpha\alpha}+O(v)\right)^{2}+\frac{2}{v^{4}}\sum_{\alpha=2}^{n}\left(-\frac{\cos\theta}{S}d_{\alpha}u_{\alpha\alpha}+O(v)\right)^{2}\right]\nonumber\\
&=&O(\frac{1}{v^{4}})\sum_{\alpha=2}^{n}|u_{\alpha\alpha}|^{2}+O(\frac{1}{v^{2}}).\nonumber
\end{eqnarray*}
Finally, we handle the other remaining terms together in \eqref{summary888} as follows
\begin{eqnarray}
I_{4}+I_{5}+I_{6}&=&-2\cos\theta G^{ij}u_{ki}d_{kj}-2KG^{ij}d_{i}v_{j}-KvG^{ij}d_{ij}\label{456}\\
&&-\cos\theta G^{ij}u_{k}d_{kij}-G_{p_{i}}(Kvd_{i}+\cos\theta u_{k}d_{ki})\nonumber\\
&=&O(\frac{1}{v})\sum_{\alpha=2}^{n}u_{\alpha\alpha}+O(v).\nonumber
\end{eqnarray}
Applying the AM-GM inequality, we compute
\begin{eqnarray}
I_{12}+I_{22}+I_{32}&=&-\frac{1}{\rho v}u_{1}S\sum_{\alpha=2}^{n}u_{\alpha\alpha}-(1-\varepsilon)\frac{(1+Kd)}{\rho v^{2}e^{\omega}}\sum_{\alpha=2}^{n}u^{2}_{\alpha\alpha}\label{equ-nonc-45644334}\\
&\leq&\frac{(n-1)}{\delta_{1}}\left(\frac{u_{1}S}{\rho v}\right)^{2}+\frac{1}{4}\delta_{1}\sum_{\alpha=2}^{n}u^{2}_{\alpha\alpha}-(1-\varepsilon)\frac{(1+Kd)}{\rho v^{2}e^{\omega}}\sum_{\alpha=2}^{n}u^{2}_{\alpha\alpha}\nonumber\\
&=&\frac{(n-1)}{4}\frac{\rho v^{2}e^{\omega}}{(1-\varepsilon)(1+Kd)}\frac{u^{2}_{1}S^{2}}{\rho^{2}v^{2}}\nonumber\\
&\leq&\frac{(n-1)}{4}\frac{e^{\omega}}{(1-\varepsilon)\rho}(1+|\cos\theta|)S u^{2}_{1},\nonumber
\end{eqnarray}
where the third equality we have taken $\delta_{1}=\frac{4(1-\varepsilon)(1+Kd)}{\rho v^{2} e^{\omega}}$; the last inequality we used fact that $S\leq(1+Kd)(1+|\cos\theta|)$.

By our hypothesis, we can choose a constant $a_{0}\in\left(|\cos \theta|, \frac{3n+1}{5n-1}\right)$. If $u$ satisfies
\begin{equation*}
  \frac{|\nabla u|^2}{v}-\cos \theta \sin \beta u_\beta<\left(1-a_0\right) u_1,
\end{equation*}
then we have
$$
\frac{|\nabla u|^2}{v}-|\cos \theta| u_1 \leq \frac{|\nabla u|^2}{v}-\cos \theta \sin \beta u_\beta<\left(1-a_0\right) u_1,
$$
which implies
$$
|\nabla u|^2 \leq \frac{\left(\cos \theta+\left(1-a_0\right)\right)^2}{1-\left(\cos \theta+\left(1-a_0\right)\right)^2},
$$
and hence we have done. Therefore, we next assume that
\begin{equation}\label{kaoooaa}
\frac{|\nabla u|^2}{v}-\cos \theta \sin \beta u_\beta \geq\left(1-a_0\right) u_1.
\end{equation}
Note that
\begin{equation}\label{lomh}
  I_{13}=\frac{(1+Kd)}{v}|\nabla u|^{2}\left(-\frac{n}{\rho v}|\nabla u|^{2}+\frac{n\cos\theta}{\rho}\sin\beta u_{\beta}\right)
\end{equation}
and
\begin{equation}\label{lomha}
  I_{23}=\cos\theta\sigma(\nabla u,\nabla d)\left(-\frac{n}{\rho v}|\nabla u|^{2}+\frac{n\cos\theta}{\rho}\sin\beta u_{\beta}\right).
\end{equation}
Hence by \eqref{kaoooaa} we get
\begin{eqnarray}
  I_{13}+I_{23}&=&-\frac{n S}{\rho}u_{1}\left(\frac{|\nabla u|^{2}}{v}-\cos\theta\sin\beta u_{\beta}\right).\label{eq-111oo}\\
  &\leq&-\frac{nS}{\rho}(1-a_{0})u^{2}_{1}.\nonumber
\end{eqnarray}
Since $a_{0}\in\left(|\cos \theta|, \frac{3n+1}{5n-1}\right)$, we choose $\varepsilon=\frac{\varepsilon_{0}}{2}$ and
\begin{equation}\label{zaaaaa}
\varepsilon_0:=\frac{3 n+1-a_0(5 n-1)}{4 n\left(1-a_0\right)}\in(0,1).
\end{equation}
Combining \eqref{equ-nonc-45644334} with \eqref{eq-111oo} we see
\begin{eqnarray}
\qquad(I_{12}+I_{22}+I_{32})+(I_{13}+I_{23})&\leq&\frac{Su^{2}_{1}}{\rho}\left(\frac{(n-1)}{4(1-\varepsilon)}e^{\omega}(1+|\cos\theta|)-n(1-a_{0})\right)\\
&\leq&\frac{Su^{2}_{1}}{\rho}\left(\frac{n-1}{4(1-\varepsilon)}(1+a_{0})-n(1-a_{0})\right)\nonumber\\
&\leq&-c_{0}u^{2}_{1},\nonumber
\end{eqnarray}
where we used the fact that $e^{\omega}=\frac{1}{x_{n+1}}\leq1$ due to $\Sigma\subseteq \mathcal{H}_{+}$ and $c_0$ is a positive constant depending only on $n, a_0$ and the $C^{0}$ normal of $u$ by \eqref{zaaaaa}.

Putting all above terms $I_{1}-I_{6}$ into \eqref{summary888}, we obtain
\begin{eqnarray}
0&\leq&I_{1}+I_{2}+I_{3}+I_{4}+I_{5}+I_{6}\nonumber\\
&=&(I_{11}+I_{14})+(I_{21}+I_{24})+I_{31}+I_{33}+(I_{12}+I_{22}+I_{32}+I_{13}+I_{23})+I_{4}+I_{5}+I_{6}\nonumber\\
&\leq&O(\frac{1}{v})\sum_{\alpha=2}^{n}u_{\alpha\alpha}-\frac{\varepsilon_{0}}{2}\frac{(1+Kd)}{\rho v^{2}e^{\omega}}\sum_{\alpha=2}^{n}u^{2}_{\alpha\alpha}+O(v)-c_{0}u^{2}_{1}\nonumber\\
&\leq&\frac{1}{\tilde{\delta}}O(\frac{1}{v^{2}})+\left(4\tilde{\delta}(n-1)-\frac{\varepsilon_{0}}{2}\frac{(1+Kd)}{\rho v^{2}e^{\omega}}\right)\sum^{n}_{\alpha=2}u^{2}_{\alpha\alpha}+O(v)-c_{0}u^{2}_{1}\nonumber\\
&\leq&-c_{0}u^{2}_{1}+O(v),\nonumber
\end{eqnarray}
where the fourth inequality we used the AM-GM inequality and the last inequality we chose $\tilde{\delta}=\frac{\varepsilon_{0}(1+Kd)}{8(n-1)\rho v^{2}e^{\omega}}$. It yields that
\begin{equation}\label{zxou}
  |\nabla u| \leq C,
\end{equation}
where the positive constant $C$ depending only on the initial data. The proof is finished.
\end{proof}

From Proposition \ref{prop3.1} and Proposition \ref{proposition3.2}, we can obtain the following uniform height and gradient estimates for the scalar parabolic equation \eqref{scalflow}.
\begin{prop}\label{propc2}If $u: \bar{\mathbb{S}}_{+}^n(E_{n+1}) \times[0, T) \rightarrow \mathbb{R}$ solves \eqref{scalflow} and $G$ has the form \eqref{inter-flow}, further assume that $|\cos \theta|<\frac{3 n+1}{5 n-1}$, then
\begin{equation}\label{jjjjj}
  \|u\|_{C^1\left(\bar{\mathbb{S}}^{n}_{+}(E_{n+1})\times[0, T)\right)} \leq C,
\end{equation}
where the constant $C$ is a positive constant depending on the initial data.
\end{prop}

\begin{prop} The solution of flow \eqref{inter-flow} exists for all time and has uniform $C^{\infty}$-estimates, if the initial hypersurface $\Sigma_0$ in a horoball $\bar{\mathcal{H}}_{+}$ is starshaped in the sense of Definition \ref{starshaped} and $|\cos \theta|<\frac{3 n+1}{5 n-1}$.
\end{prop}
\begin{proof}
From the a priori estimates in Proposition \ref{propc2}, we imply that $u$ is uniformly bounded in $C^1(\mathbb{S}_{+}^n(E_{n+1}) \times[0, T))$ and the scalar equation in \eqref{scalflow} is uniformly parabolic. Since $|\cos \theta|<1$, then the long time existence and uniform high order estimates follows from the quasilinear parabolic PDE theory with strictly oblique boundary condition (see e.g \cite[Theorem 13.16]{Lieb} and \cite[Chapter IV, Theorem 5.3]{LSU}).
\end{proof}
Finally, we show the convergence and uniqueness results by applying the argument in \cite{SWX, WengX} and hence complete the proof of Theorem \ref{mainthm}.

\begin{prop} If the initial hypersurface $\Sigma$ in a horoball $\mathcal{H}_{+}$ is star-shaped capillary boundary hypersurface and $|\cos \theta|<\frac{3 n+1}{5 n-1}$, then the flow \eqref{mainflow} smoothly converges to a uniquely determined umbilical hypersurface $C_{\theta, r}(E_{n+1})$ given by \eqref{model1} whose enclosed volume
is the same as that of $x_{0}$.
\end{prop}
\begin{proof}
The $C^{0}$ and $C^{1}$ estimates imply that $\bar{g}\left(X_{n+1}, \nu\right) \geq c>0$, i.e the star-shapedness is preserved along the flow \eqref{mainflow}.

%Since we have the uniform $C^{0}, C^{1}$ estimate in \eqref{C0} and \eqref{C1}, and \eqref{scalflow} is a divergent type quasi-linear parabolic equation, By the classical theory of parabolic equation in divergent form, we can get the higher order priori estimates. Hence we prove that \eqref{scalflow} has a smooth solution for all time with uniform $C^{\infty}$-estimate. Next we integrate (1.7) over $t \in[0,+\infty)$ and using the uniform estimates, we obtain

From \eqref{energy-1} we know energy functional $\mathcal{E}(t)$ is non-increasing and
\begin{equation}\label{energy-2223}
\frac{d}{d t} \mathcal{E}(t)=-\frac{1}{n-1} \int_{\Sigma_{t}} \sum_{1 \leq i<j \leq n}\left(\kappa_{i}-\kappa_{j}\right)^{2} \bar{g}\left(X_{n+1}, \nu\right) d A_{t}\leq0,
\end{equation}
where $\kappa_{i}, i=1, \cdots, n$ are the principal curvatures of $\Sigma_{t}$.

It follows from the long time existence and uniform $C^{\infty}$-estimates, we see
\begin{equation}\label{energy-2223ww}
-\int_{0}^{\infty}\frac{d}{d t} \mathcal{E}(t)dt=\mathcal{E}(0)-\mathcal{E}(\infty)<+\infty.
\end{equation}
Hence we get
\begin{equation}\label{energy-222w5}
\int_{\Sigma_{t_{k}}} \sum_{1 \leq i<j \leq n}\left(\kappa_{i}-\kappa_{j}\right)^{2} \bar{g}\left(X_{n+1}, \nu\right) d A_{t}\rightarrow0, \,\,\text{as}\,\, t_{k}\rightarrow+\infty.
\end{equation}
Due to $\bar{g}(X_{n+1}, \nu)\geq c>0$, then there exists a convergent subsequence such that
\begin{equation}\label{energy-2dw5}
\sum_{1 \leq i<j \leq n}\left(\kappa_{i}-\kappa_{j}\right)^{2} \bar{g}\left(X_{n+1}, \nu\right)=0, \,\,\text{as}\,\, t_{k_{l}}\rightarrow+\infty.
\end{equation}
Therefore $x(\cdot, t)$ smoothly subconverges to an umbilical hypersurface with $\theta$-capillary boundary.

Finally we show that the above limit umbilical hypersurface is unique. We follow an approach in \cite{WWX, WengX}. Denote $C_{\theta, r_{\infty}}(E_{\infty})$ be the umbilical hypersurface with radius $r_{\infty}$ around $E_\infty$ in $\mathcal{H}_{+}$ as the limits of $x(\cdot, t)$, and denote $r(\cdot, t)$ be the radius of the unique umbilical hypersurface around $E_{n+1}$ passing through the point $x(\cdot, t)$.

Note that the volume of $\widehat{C_{\theta, r}(E_{n+1})}$, denoted by $V(\theta, r)$, is strictly increasing function with respect to $r$. Indeed, we can see the following flow
\begin{eqnarray*}
\partial_{t}{y}&=&X_{n+1}\\
y(0,\cdot)&=&C_{\theta,r_{0}}(E_{n+1}).
\end{eqnarray*}
The flow hypersurfaces of this flow are $C_{\theta,r(t)}(E_{n+1})$, where $r(t)$ is an increasing function and satisfies
\begin{equation*}
  r'(t)=\frac{2x_{n+1}r^{2}\bar{g}(X_{n+1},\nu)}{|x_{n+1}-E_{n+1}|^{2}+r^{2}\sin^{2}\theta}>0 \,\,\,\,\quad \text{with} \,\,\,r(0)=r_{0}.
\end{equation*}
Hence there holds
\begin{equation*}
  r'(t)\frac{\partial}{\partial r}V(\theta, r(t))=\frac{\partial}{\partial t}V(\theta, r(t))=\int_{C_{\theta,r(t)}}\bar{g}(X_{n+1},\nu)dA>0.
\end{equation*}

Since the enclosed volume is preserved along the flow \eqref{mainflow} and it is strictly monotone with respect to the radius $r$, thus $r_{\infty}$ is independent of the choice of the subsequence of $t$.

Next we will prove the uniqueness, i.e. $E_{\infty}=E_{n+1}$. Due to the barrier estimate from Proposition \ref{prop3.1}, we see
\begin{equation}\label{barrier}
  r_{\max }(t):=\max r(\cdot\,, t)=r\left(\xi_{t}, t\right),
\end{equation}
is non-increasing with respect to $t$, which implies $\lim\limits_{t\rightarrow+\infty} r_{\max}(t)$ exists.

Now we claim that
\begin{equation}\label{eq124}
  \lim_{t \rightarrow+\infty} r_{\max }(t)=r_{\infty}.
\end{equation}
By contradiction. If not, then there exists $\epsilon>0$ such that
\begin{equation}\label{eq125}
  \frac{1}{r_{\max}(t)}<\frac{1}{r_{\infty}}-\frac{\epsilon}{n},
\end{equation}
for $t$ large enough. Note that $r(\xi_{t}, t)$ satisfies
\begin{equation}\label{jbvg}
  2(1-r\cos\theta)\langle x,E_{n+1}\rangle=|x|^{2}+1-2r\cos\theta-r^{2}\sin^{2}\theta.
\end{equation}
By taking the time $t$ derivative for \eqref{jbvg}, we have
\begin{eqnarray}\label{eq678}
(r\sin^{2}\theta+\cos\theta-\cos\theta\langle x,E_{n+1}\rangle)\partial_{t}r=\langle\partial_{t}x,x-(1-r\cos\theta)E_{n+1}\rangle.
\end{eqnarray}

In the following we calculate at point $\left(\xi_t, t\right)$. Since $\Sigma_t$ is tangential to $C_{\theta, r_{\max}}(E_{n+1})$ at $\left(\xi_t, t\right)$, which implies
\begin{alignat}{2}
\left(\nu\right)_{\Sigma_{\mathrm{t}}}\left(\xi_{\mathrm{t}}, t\right)=\left(\nu\right)_{C_{\theta, r_{\max }(E_{n+1})}}=x_{n+1}\cdot\frac{x-(1-r_{\max}(t)\cos\theta )E_{n+1}}{r_{\max}(t)}.\label{jsjsj}
\end{alignat}
Utilizing \eqref{eq678}-\eqref{jsjsj} and the flow \eqref{mainflow}, we see
\begin{eqnarray}\label{eq678wdsss}
&&(r_{\max}(t)\sin^{2}\theta+\cos\theta-x_{n+1}\cos\theta)\partial_{t}r_{\max}(t)\\
&=&\langle\partial_{t}x, x-(1-r_{\max}\cos\theta)E_{n+1}\rangle\nonumber\\
&=&f\langle\nu, x-(1-r_{\max}\cos\theta)E_{n+1}\rangle\nonumber\\
&=&x_{n+1}r_{\max}(t)(nV_{n+1}-n\cos\theta \bar{g}(x,\nu)-H\bar{g}(X_{n+1},\nu)).\nonumber
\end{eqnarray}
From \eqref{Cr}, on $C_{\theta, r_{\max}}(E_{n+1})$ we have
\begin{equation}\label{static1}
\frac{V_{n+1}-\cos\theta\bar{g}(x,\nu)}{\bar{g}\left(X_{n+1}, \nu\right)}=\kappa=\left(\frac{1}{r_{\max }}-\cos\theta\right),
\end{equation}
where $\kappa$ is the principal curvature of $C_{\theta, r_{\max}}(E_{n+1})$. Since $x(\cdot, t)$ and ${C}_{\theta, r_{\max }}(E_{n+1})$ have the same normal vector at $(\xi_{t}, t)$, we obtain
\begin{equation}\label{eq456}
  (r_{\max}(t)\sin^{2}\theta+\cos\theta-x_{n+1}\cos\theta)\partial_{t}r_{\max}(t)=x_{n+1}r_{\max}(t)\left(\frac{n}{r_{\max}}-n\cos\theta-H\right)\bar{g}(X_{n+1},\nu).
\end{equation}
Note that there exists some constant $c_{1}>0$ such that
\begin{equation}\label{uyhhs}
  r_{\max}(t)\sin^{2}\theta+\cos\theta-x_{n+1}\cos\theta\geq c_{1}>0.
\end{equation}
In fact, by \eqref{jbvg} we have
\begin{equation}\label{kiooq-9}
  (\cos\theta-x_{n+1}\cos\theta)=\frac{|x-E_{n+1}|^{2}-r^{2}\sin^{2}\theta}{2r}.
\end{equation}
Therefore,
\begin{eqnarray}
(r\sin^{2}\theta+\cos\theta-x_{n+1}\cos\theta)&=&r\sin^{2}\theta+\frac{|x-E_{n+1}|^{2}-r^{2}\sin^{2}\theta}{2r}\\
&=&\frac{1}{2r}(|x-E_{n+1}|^{2}+r^{2}\sin^{2}\theta)\nonumber\\
&\geq&\frac{r}{2}\sin^{2}\theta>0.\nonumber
\end{eqnarray}
Since $x(\cdot, t)$ converges to ${C}_{\theta, r_{\infty}}\left(E_{\infty}\right)$ and $r_{\infty}$ is uniquely determined, we have
\begin{equation*}
  H \rightarrow n\left(\frac{1}{r_{\infty}}-\cos\theta\right)\,\,\,\,\, \text { uniformly,}
\end{equation*}
that is, there exists $T_{0}>0$ such that for any $t>T_{0}$, it holds
\begin{equation}\label{eq234}
  n\left(\frac{1}{r_{\infty}}-\cos\theta\right)-H<\frac{\epsilon}{2} .
\end{equation}
Combining \eqref{eq234} with \eqref{eq125}, we get
\begin{equation}\label{axmhhh}
  n\left(\frac{1}{r_{\max}(t)}-\cos\theta\right)-H< n\left(\frac{1}{r_{\infty}}-\cos\theta\right)-H-\epsilon<-\frac{\epsilon}{2},
\end{equation}
for any $t>T_{0}$. From Proposition \ref{prop3.1} we know that $x^{2}_{n+1}\bar{g}^{2}(X_{n+1},\nu)$ in \eqref{eq456} has uniformly positive lower bound.
Then for any $t>T_{0}$, we obtain
\begin{equation}
  \partial_{t} r_{\max}(t) \leq-C \epsilon,
\end{equation}
where $C$ is some universal positive constant. However, we get a contradiction by the fact that $\lim\limits_{t \rightarrow+\infty} \partial_{t} r_{\max}(t)=0$. The claim \eqref{eq124} is true. Similarly, we can also show
\begin{equation}\label{rminimal}
\lim _{t \rightarrow+\infty} r_{\min}(t)=r_{\infty} .
  \end{equation}
In conclusion, from \eqref{rminimal} and \eqref{eq124}, we see that $\lim\limits _{t \rightarrow +\infty} r(\cdot, t)=r_{\infty}$ and hence the uniqueness has been proved. We complete the proof.
\end{proof}

\

{\bf Acknowledgements.} The work was supported by China Postdoctoral Science Foundation (No.2022M720079) and Shuimu Tsinghua Scholar Program (No.2022SM046). The author thanks Prof. Haizhong Li, Prof. Guofang Wang and Prof. Chao Xia for their constant support and their interest on this topic.

\

\end{document}